\documentclass[11pt,a4paper]{amsart}

\usepackage{amssymb}
\usepackage[T1]{fontenc}
\usepackage{mathpazo}

\usepackage{mathrsfs}

\usepackage[all]{xy}
\usepackage[headings]{fullpage}
\usepackage{paralist}

\usepackage{color}

\usepackage[colorlinks=true,linkcolor=blue, citecolor=blue,%
urlcolor=blue, pagebackref=true]{hyperref}

\makeatletter

\@addtoreset{equation}{section}
\def\theequation{\thesection.\@arabic \c@equation}

\def\@citecolor{blue}
\def\@linkcolor{blue}
\def\@urlcolor{blue}
\def\theenumi{\@alph\c@enumi}

\makeatother

\theoremstyle{plain}
\newtheorem{theorem}[equation]{Theorem}
\newtheorem{lemma}[equation]{Lemma}

\newtheorem{proposition}[equation]{Proposition}

\theoremstyle{definition}

\newtheorem{remark}[equation]{Remark}
\newenvironment{remarkbox}[1][]{%
    \begin{remark}[#1] \pushQED{\qed}}{\popQED \end{remark}}

\newtheorem{example}[equation]{Example}
\newenvironment{examplebox}[1][]{%
    \begin{example}[#1] \pushQED{\qed}}{\popQED \end{example}}

\newtheorem{definition}[equation]{Definition}

\newtheorem{notation}[equation]{Notation}

\newtheorem{discussion}[equation]{Discussion}
\newenvironment{discussionbox}[1][]{%
    \begin{discussion}[#1]\pushQED{\qed}}{\popQED \end{discussion}}

\newtheorem{observation}[equation]{Observation}
\newenvironment{observationbox}[1][]{%
    \begin{observation}[#1]\pushQED{\qed}}{\popQED \end{observation}}
\newcommand{\fraka}{{\mathfrak a}}
\newcommand{\calF}{\mathcal F}
\newcommand{\calI}{\mathcal I}
\newcommand{\frakM}{{\mathfrak M}}
\newcommand{\frakm}{{\mathfrak m}}

\newcommand{\frakN}{{\mathfrak N}}
\newcommand{\frakn}{{\mathfrak n}}

 \let\strSh\calO
\newcommand{\frakQ}{{\mathfrak Q}}
\newcommand{\naturals}{\mathbb{N}}
\newcommand{\ints}{\mathbb{Z}}
\def\to{\longrightarrow}
\newcommand{\AssGr}{\mathrm{gr}}
\newcommand{\Rees}{\mathscr{R}}
\DeclareMathOperator{\height}{ht}
\DeclareMathOperator{\Hom}{Hom}
\DeclareMathOperator{\Spec}{Spec}
\DeclareMathOperator{\Proj}{Proj}

\DeclareMathOperator{\homology}{H}
\newcommand{\define}[1]{\emph{#1}}
\newcommand{\minus}{\ensuremath{\smallsetminus}}
\DeclareMathOperator{\image}{Im}
\DeclareMathOperator{\Cech}{\check C}
\newcommand{\tCech}{\v Cech\,}
\DeclareMathOperator{\Ann}{Ann}
\DeclareMathOperator{\Ext}{Ext}
\makeatletter
\def\RDerChar{\mathbf{R}}
\def\RDer{\@ifnextchar[{\R@Der}{\ensuremath{\RDerChar}}}
\def\R@Der[#1]{\ensuremath{\RDerChar^{#1}}}

\makeatother

\begin{document}

\title{$F$-rationality of Rees Algebras}

\author{Mitra Koley}
\address{Chennai Mathematical Institute, Siruseri, Tamilnadu 603103. India}
\curraddr{Tata Institute of Fundamental Research, Mumbai, Maharashtra
400005. India}
\email{mitra@cmi.ac.in}
\author{Manoj Kummini}
\address{Chennai Mathematical Institute, Siruseri, Tamilnadu 603103. India}
\email{mkummini@cmi.ac.in}

\thanks{Both the authors were partially supported by a grant from the Infosys Foundation.}


\dedicatory{To Prof. Craig Huneke}
\maketitle

\section{Results}

Let $R$ be a noetherian ring and $I$ an $R$-ideal.
The \define{Rees algebra} of $R$ with respect to $I$ is 
$\Rees_{R}(I) := \oplus_{n \geq 0} I^n$;
the \define{extended Rees algebra} of $R$ with respect to $I$ is 
$\Rees'_{R}(I) := \oplus_{n \in \ints} I^n$, where, for $n \leq 0$ 
$I^n := R$.

Several authors have studied the singularities of Rees algebras:
E.~Hyry~\cite{HyryBlowupRingsRationalSings1999} for rational singularties
in characteristic zero;
A.~K.~Singh~\cite{SinghMultiSymbReesAndStrongFrationality2000}
in prime characteristic for strong $F$-regularity and $F$-purity; and 
N.~Hara,
K.-i.~Watanabe and K.-i.~Yoshida~\cite{HaraWatanabeYoshidaFrationality2002}
and\cite{HaraWatanabeYoshidaReesAlgFregular02}
in prime characteristic for $F$-rationality and $F$-regularity 
respectively.

In this article, we study the $F$-rationality of  
$\Rees_{R}(I)$ and $\Rees'_{R}(I)$. Our primary aim is to understand some
questions of N.~Hara, K.-i.~Watanabe and 
K.-i.~Yoshida~\cite[Section~1]{HaraWatanabeYoshidaFrationality2002}, which
ask for necessary and sufficient conditions for Rees algebras to be
$F$-rational. We list our results below, postponing definitions to
Section~\ref{section:prelims}.
\begin{theorem}
\label{theorem:extReesAndRees}
Let $(R, \frakm)$ be an excellent local
domain of prime characteristic and $I$ an $\frakm$-primary
$R$-ideal. Then the following are equivalent:
\begin{asparaenum}
\item 
\label{theorem:extReesAndRees:ExtReesFratl}
$\Rees'_{R}(I)$ is $F$-rational;
\item
\label{theorem:extReesAndRees:ReesFratl}
$R$ and $\Rees_R(I)$ are $F$-rational.
\end{asparaenum}
\end{theorem}

This settles~\cite[Conjecture~4.1]{HaraWatanabeYoshidaFrationality2002}
which asserted the conclusion of the above theorem. That the
$F$-rationality of $\Rees_R(I)$ implies the $F$-rationality of
$\Rees'_R(I)$ is~\cite[Theorem~4.2]{HaraWatanabeYoshidaFrationality2002},
but we give a different proof, which follows directly from some
observations on the tight closure of zero in the local cohomology modules
of  $\Rees_R(I)$ and of $\Rees'_R(I)$ and on the $F$-rationality of 
$\Proj \Rees_R(I)$ that we discuss in Section~\ref{section:rees}. 
As applications of the results of Section~\ref{section:rees}, 
we get a sufficient (but not necessary) condition for the $F$-rationality
of $R$ given the $F$-rationality of $\Rees_R(I)$
(Proposition~\ref{theorem:assGrFInj}) and 
recover~\cite[Theorem~3.1]{HaraWatanabeYoshidaFrationality2002} about 
the $F$-rationality of Rees algebras of integrally closed ideals in
two-dimensional $F$-rational rings (Theorem~\ref{theorem:dimtwo}).

Our next result partially
answers~\cite[Question~3.7]{HaraWatanabeYoshidaFrationality2002}, which
asked whether the result holds (for $\Rees_R(I)$), without any restriction on the dimension.
Since our proof uses the principalization result of V.~Cossart and
O.~Piltant~\cite[Proposition~4.2]{CossartPiltantReslI2008}, we put some
conditions on $R$.

\begin{theorem}
\label{theorem:rationalblowup}
Let $R$ be a three-dimensional finite-type domain over a 
field of prime characteristic and $\frakm$
a maximal ideal. Assume that $R$ is a rational singularity. Let $I$ be an
$\frakm$-primary ideal. 
Let $S$ be a graded $R$-algebra with 
${\Rees_{R}(I)} \subseteq S \subseteq \overline{\Rees_{R}(I)}$. 
Suppose that $\Proj S$ is
$F$-rational. Then the ring $\oplus_{n \geq 0} S_{Nn}$ is
$F$-rational for every integer $N \gg 0$. 
\end{theorem}
This paper arose from trying to understand 
whether the results of Hyry~\cite{HyryBlowupRingsRationalSings1999} 
(who, in characteristic zero, relate the
rationality of a Rees algebra to that of the corresponding blow-up) have
counterparts for $F$-rationality.

In Section~\ref{section:prelims}, we give the definitions, some known
results needed in our proofs and some preliminary lemmas.
The subsequent sections contain the proofs of the above theorems.

\subsection*{Acknowledgements}
We thank the referee for his/her comments.

\section{Preliminaries}
\label{section:prelims}

For the remainder of this paper, unless otherwise indicated, all rings
considered are excellent, and of prime characteristic $p$. For a ring $S$,
the \define{Frobenius endomorphism}, denoted $F_S$ (or, merely $F$, if no
confusion is likely to arise), is the map $s \mapsto s^p$. When used in the
context of Frobenius endomorphisms, tight closure, $F$-rationality etc, $q$
denotes a power of $p$, and the expression ``$q \gg 0$'' is synonymous with
``$q = p^e$ for every integer $e \gg 0$.'' By $S^0$, we mean the complement
in $S$ of the union of minimal primes of $S$.
For a ring $S$, $\overline{S}$ denotes its normalization.

Throughout this paper, $(R, \frakm)$ is a $d$-dimensional ring with $d \geq
2$, $I$ an 
$\frakm$-primary ideal, admitting a
reduction $J = (f_1, \ldots, f_d)$. We write 
$\Rees := \Rees_R(I)$, $\Rees' :=
\Rees'_R(I)$, $Jt = (f_1t, \ldots, f_dt)\Rees$.
$\frakM := \frakm\Rees + \Rees_+$.
$\frakN := \sqrt{\Rees_+\Rees' + (t^{-1})\Rees'}$ and 
$X = \Proj \Rees$. Write
$E$ for the exceptional divisor of $X$, i.e., the effective Cartier divisor
defined by $I\strSh_X$.
Write $G = \AssGr_I(R)$ and $G_+$ for the $G$-ideal $\oplus_{n > 0} G_n$.
Note that $E = \Proj G$.

\subsection*{Local cohomology}
We discuss some properties of local cohomology and the Frobenius action on
it.

\begin{discussionbox}
\label{discussionbox:HomologyOfRees}
Using %
the two exact sequences
\[
0 \to \Rees_+(1) \to \Rees \to G \to 0 \quad\text{and}\quad
0 \to \Rees_+ \to \Rees \to R \to 0,
\]
and the fact that $\frakM G = \sqrt{G_+}$ and 
$\frakM R = \frakm$, one can conclude that 
$\homology^{d+1}_{\frakM}(\Rees)_n = 0$ for every $n \geq 0$ and that 
$\homology^{d+1}_{\frakM}(\Rees)_{-1} \neq 0$; see,
e.g.,~\cite[Lemma~3.3, p.~87]{GotoNishidaFiltrGorRees1994}. A similar
conclusion holds also for a graded $R$-algebra with $\Rees \subseteq S
\subseteq \overline{\Rees}$.

Note that $\homology^i_{\Rees_+}(R) = 0$ for every $i>0$. Hence 
the two exact sequences above give an exact sequence
\[
\xymatrix{
\cdots \ar[r] &  
\homology^2_{\Rees_+}(\Rees)(1) \ar[r] & 
\homology^2_{\Rees_+}(\Rees) \ar[r] & 
\homology^2_{G_+}(G) 
\ar `r `[l] `[llld] `[d] [lld]\\
& & \cdots&   
\ar `r `[l] `[llld] `[d] [lld]\\
&
\homology^d_{\Rees_+}(\Rees)(1) \ar[r] & 
\homology^d_{\Rees_+}(\Rees) \ar[r] & 
\homology^d_{G_+}(G)  \ar[r] & 0.
}
\]
For $i \geq 2$, $\homology^i_{\Rees_+}(\Rees)_j = 
\homology^{i-1}(X, (I\strSh_X)^{j})$, so 
$\homology^i_{\Rees_+}(\Rees)_j = 0 $ for every $j \gg 0$.
\end{discussionbox}

\begin{remarkbox}
\label{remarkbox:CMRees}
Suppose that $R$ is Cohen-Macaulay. Then 
$\Rees'$ is Cohen-Macaulay if and only if $G$ is Cohen-Macaulay.
Further,
$\Rees$ is Cohen-Macaulay if and
only if $G$ is Cohen-Macaulay and $\left(\homology^d_{G_+}(G)\right)_n = 0$
for every $n \geq 0$. See~\cite[Theorem~1.1]{GSrees82}, along with 
Remark~3.10 on p.~218 and the discussion surrounding (*) and (**) on
pp.~202--203 in~\cite{GSrees82}. Observe from
Discussion~\ref{discussionbox:HomologyOfRees} that
$\left(\homology^d_{G_+}(G)\right)_n = 0$
for every $n \geq 0$ if and only if
$\homology^i_{\Rees_+}(\Rees)_n = 0 $ for every $n \geq 0$.
\end{remarkbox}

\begin{remarkbox}
\label{remarkbox:cechcyclezero}
Let $(S, \frakn)$ be an $n$-dimensional Cohen-Macaulay ring and $x_1,
\ldots, x_n$ a system of parameters. Write $x = x_1\cdots x_n$. Let $a \in
S$. Then the \tCech cycle
\[
\frac{a}{x^l} 
\]
gives the zero element of $\homology^n_\frakn(S)$ if and only if $a \in
(x_1^l, x_2^l, \ldots, x_n^l)$; 
see~\cite[Proof of Theorem~2.1, pp.~104--105]{LipTeiPseudoRatlSing81}.
\end{remarkbox}

\begin{observationbox}
\label{observationbox:topCohGradedRing}
Let $S := \oplus_{n \in \naturals}S_n$ be a $d$-dimensional graded ring.
Suppose that there exists 
an $S$-regular element $x \in S_1$. Then
we
have an exact sequence
\[
\cdots \to \homology^{d-1}_{S_+}(S/(x)) \to 
\homology^{d}_{S_+}(S)(-1) \to 
\homology^{d}_{S_+}(S) \to 0.
\]
From this it follows that
for every $j \in \ints$, if 
$\left(\homology^d_{S_+}(S)\right)_j = 0$ then 
$\left(\homology^d_{S_+}(S)\right)_{j+1} = 0$.
\end{observationbox}

The Frobenius map on a ring $S$ commutes with
the localization map $S \to S_a$
for every $a \in S$. Therefore, for every sequence $a_1, \ldots, a_n \in
S$, we get a map of the \tCech complex 
$\Cech^\bullet(a_1, \ldots, a_n; S)$ and hence of the local cohomology
modules $\homology^*_{(a_1, \ldots, a_n)}(S)$. We now discuss some
properties of this action.

\begin{discussionbox}
\label{discussionbox:FrobOnIdeals}
Let $S$ be a ring, $x \in S$ a non-zero-divisor on $S$, 
and $\fraka = (a_1, \ldots, a_n)$. We have the following commutative
diagram of abelian groups, in which the rows are exact sequences of
$S$-modules:
\[
\xymatrix{
0 \ar[r] & 
S \ar[r]^{x} \ar[d]_{x^{p-1}F}&
S \ar[r] \ar[d]^F&
S/(x) \ar[r] \ar[d]^F& 0 \\
0 \ar[r] & 
S \ar[r]^{x} &
S \ar[r] &
S/(x) \ar[r] & 0 \\ }
\]
These maps commute with localization, thus giving us the following
commutative diagram (of abelian groups, in which the rows are exact
sequences of $S$-modules):
\[
\xymatrix{
0 \ar[r] & 
\Cech^\bullet(a_1, \ldots, a_n;S) \ar[r]^x \ar[d]_{x^{p-1}F}&
\Cech^\bullet(a_1, \ldots, a_n;S) \ar[r] \ar[d]^F&
\Cech^\bullet(a_1, \ldots, a_n;S/(x)) \ar[r] \ar[d]^F& 0 \\
0 \ar[r] & 
\Cech^\bullet(a_1, \ldots, a_n;S) \ar[r]^x &
\Cech^\bullet(a_1, \ldots, a_n;S) \ar[r]&
\Cech^\bullet(a_1, \ldots, a_n;S/(x)) \ar[r]& 0 \\}
\]
This gives a corresponding commutative diagram of local cohomolgy modules
$\homology^*_\fraka(-)$. We observe that even though the vertical maps are
maps of abelian groups, their kernels are $S$-modules. 
Now assume that $S$ is graded and $x$ is a homogeneous element of degree
$m$. This gives the following commutative diagram of local cohomology modules
\begin{equation}
\label{equation:FrobModNZD}
\vcenter{\vbox{%
\xymatrix@C=2ex{
\cdots \ar[r] &
\homology^{i-1}_\fraka(\frac S{(x)})_j \ar[r] \ar[d]^F&
\homology^{i}_\fraka(S)_{-m+j} \ar[r]^x \ar[d]^{x^{p-1}F}&
\homology^{i}_\fraka(S)_{j} \ar[r] \ar[d]^F&
\homology^{i}_\fraka(\frac S{(x)})_j \ar[r] \ar[d]^F&
\homology^{i+1}_\fraka(S)_{-m+j} \ar[r] \ar[d]^{x^{p-1}F}&\cdots \\
\cdots \ar[r] &
\homology^{i-1}_\fraka(\frac S{(x)})_{pj} \ar[r] &
\homology^{i}_\fraka(S)_{-m+pj} \ar[r]^x&
\homology^{i}_\fraka(S)_{pj} \ar[r]&
\homology^{i}_\fraka(\frac S{(x)})_{pj} \ar[r] &
\homology^{i+1}_\fraka(S)_{-m+pj} \ar[r] &\cdots 
}}}
\end{equation}
\end{discussionbox}

\subsection*{$F$-rationality}

Let $S$ be a ring and $M$ an $S$-module.
For a positive integer $e$, write ${}^eS$ for the $S$-module $S$ considered
through the $e$th iteration of the Frobenius ring endomorphism $S \to S$.
Write ${}^eM$ for $M \otimes_S {}^eS$. For $z \in M$, $z^q$ is the image of
$z$ under the map $M \to {}^eM$. For a submodule $N$ of $M$, $N^{[q]}_M$
denotes the $S$-submodule of ${}^eM$ generated by the image of 
${}^eN \to {}^eM$. The \define{tight closure} of $N$ in
$M$ is
\[
N^*_M := 
\left\{z \in M \mid \text{there exists}\; c \in S^0 \text{such that for
all}\;  q \gg 0, cz^q \in N^{[q]}_M\right\}.
\]
We say that $N$ is \define{tightly closed in $M$} if $N^*_M = N$. 
For an $S$-ideal $I$, we write $I^* = I^*_S$ and say that $I$ is
\define{tightly closed} if it is tightly closed in $S$.

Of particular interest to us is $0^*_{\homology^d_{\frakn}(S)}$ for a
$d$-dimensional Cohen-Macaulay 
local ring $(S, \frakn)$. Let $a_1, \ldots, a_d$ 
be a system of parameters for $S$. Then for all $b \in S$, the class of the
\tCech cycle
$\frac{b}{(a_1\cdots a_d)^l}$
belongs to $0^*_{\homology^d_{\frakn}(S)}$
if and only if there exists $c \in S^0$ such that for
all  $q \gg 0$, the class of the \tCech cycle
$\frac{cb^q}{(a_1\cdots a_d)^{lq}}$ is zero in 
$\homology^d_{\frakn}(S)$.

$F$-rationality was defined by R.~Fedder and 
Watanabe~\cite{FedderWatanabeMSRI89}.
Since we work only with equi-dimensional homomorphic images of excellent
Cohen-Macaulay rings, we will use the following definition of
$F$-rationality (see~\cite[(4.2e)~and~(6.27)]{HochsterHunekeFregSmBsch94}):
\begin{definition}
A local ring is said to be \define{$F$-rational} if the ideal generated
some system of parameters is tightly closed.
A ring $S$ is said to be \define{$F$-rational} if $S_\frakn$ is
$F$-rational for every maximal $S$-ideal $\frakn$. A scheme is
\emph{$F$-rational} if the local ring at every closed point is
$F$-rational.
\end{definition}

For the rings and the schemes that we consider in this paper, if the local
ring at every maximal ideal / closed point is $F$-rational, then all the
local rings are $F$-rational, by~\cite[(4.2f)]{HochsterHunekeFregSmBsch94}.

\begin{proposition}
\label{proposition:PolyRationalFrationality}
Let $S$ be the homomorphic image of a Cohen-Macaulay ring. 
If $S$ is $F$-rational, then $S[T]$ and 
$S[T, T^{-1}]$ are $F$-rational.
\end{proposition}

\begin{proof}
The first statement
is~\cite[Proposition~3.2]{VelezOpennessFrationalLoci1995}.
Since $S[T, T^{-1}]$ is the homomorphic
image of a Cohen-Macaulay ring, it suffices to show that
$S[T, T^{-1}]_\frakn$ is $F$-rational for every maximal ideal $\frakn$ of 
$S[T, T^{-1}]$~\cite[Theorem~4.2]{HochsterHunekeFregSmBsch94}. 
This holds since $S[T]$ is $F$-rational.
\end{proof}

\begin{definition}
Let $S$ be an equi-dimensional local ring.
The \define{parameter test ideal} $\tau_p(S)$ of $S$ is the set 
of all the elements $c \in S$ such that $cI^* \subseteq I$ for every ideal
generated by part of a system of parameters. A \define{parameter test
element} is an element $c \in \tau_p(S) \cap S^0$.
\end{definition}

\begin{proposition}%
[\protect{\cite[Proposition~4.4]{SmithTestIdealsInLocalRings1995}}]
Let $(S, \frakn)$ be a Cohen-Macaulay local ring. Then 
$\tau_p(S) = \Ann_S\left(0^*_{\homology^{\dim S}_\frakn(S)}\right)$.
\end{proposition}

\begin{observationbox}
\label{observationbox:finiteLengthTCOfZero}
Let $(R, \frakm)$ be an excellent 
local
domain and $S$ a domain that is finitely generated as an
$R$-algebra.  Let $\frakn$ be a maximal ideal of $S$. Suppose 
that $\Spec S \minus \{\frakn\}$ is $F$-rational.
Then, by~\cite[Theorem~3.9]{VelezOpennessFrationalLoci1995}, for every $c
\in \frakn$, there exists a positive integer $N$ such that $c^N$ is a
parameter test element. Hence $\sqrt{\tau_p(S_\frakn)} =\frakn$
or ${\tau_p(S_\frakn)} =S_\frakn$. 
We see
from the proof 
of~\cite[Proposition 4.4(ii)]{SmithTestIdealsInLocalRings1995} that 
${\tau_p(S_\frakn)} \subseteq 
\Ann_R(0^*_{\homology^{\dim S_\frakn}_{\frakn}(S_\frakn)})$;
so 
$0^*_{\homology^{\height \frakn}_{\frakn}(S)}=
0^*_{\homology^{\dim S_\frakn}_{\frakn}(S_\frakn)}$ 
is a finite-length $S$-module.
\end{observationbox}

\subsection*{Rational singularities}
In order to prove Theorem~\ref{theorem:rationalblowup}, we need to discuss
desingularization and rational singularities. We restrict our attention to
integral schemes. A scheme $X$ is said to be \define{regular} if the local
ring $\strSh_{X,x}$ is a regular local ring for every $x \in X$.

\begin{definition}
Let $X$ be an integral scheme. A \define{desingularization} of $X$ is a
proper birational morphism $Y \to X$ such that $Y$ is regular.
\end{definition}

\begin{definition}
\label{definition:rationalsing}
Let $X$ be a scheme. We say that $X$ is a \define{rational singularity} if 
there exists a desingularization $g : Y \to X$ such that 
$\strSh_X \to \RDer g_* \strSh_{Y}$ is 
a quasi-isomorphism, i.e., 
$g_*\strSh_Y = \strSh_X$ and $\RDer^ig_*\strSh_Y = 0$ for every $i > 0$.
\end{definition}

Suppose that $X$ is Cohen-Macaulay with a dualizing sheaf $\omega_X$.
Write $\omega_Y$ for the left-most non-zero cohomology for the complex
$g^!\omega_{X}$. Then $\omega_Y$ is a dualizing sheaf for $Y$. Moreover, 
$\strSh_X \to \RDer g_* \strSh_{Y}$ is 
a quasi-isomorphism if and only if 
$\RDer g_*\omega_Y \to \omega_{X}$ is a 
quasi-isomorphism~\cite[Lemma~(4.2)]{LipCMgraded94}.

In characteristic zero, it is known that if $X$ is a rational singularity,
then \emph{every} desingularization $Y \to X$ satisfies the conditions of
Definition~\ref{definition:rationalsing}. This can be proved using
`principalization of ideal sheaves in regular schemes' and vanishing of
cohomology for finite sequence of blow-ups along regular centres
(Lemma~\ref{lemma:finiteSeqBlowupsRDer}). Principalization of ideal sheaves
is not known in general in positive characteristic, in dimensions greater
than three. We use the following
result about principalization in our proof of
Theorem~\ref{theorem:rationalblowup}.

\begin{proposition}%
[\protect{\cite[Proposition~4.2]{CossartPiltantReslI2008}}]
\label{proposition:principalization}
Let $Y$ be a three-dimensional 
quasi-projective regular variety over a field and $\calI
\subseteq \strSh_Y$ an
ideal sheaf. Then there
exists a finite sequence of morphisms
\[
Y_n \to Y_{n-1} \to \cdots \to Y_1 \to Y_0 := Y
\]
such that $Y_{i+1} \to Y_i$ is the blow-up along 
a regular subscheme of $Y_i$, for each $i$, and $I\strSh_{Y_n}$ is an
invertible sheaf.
\end{proposition}

The following lemma is perhaps well-known, but we include a proof here
for the sake of completeness.

\begin{lemma}
\label{lemma:finiteSeqBlowupsRDer}
Let $\mu$ the composite of a sequence of morphisms
\[
Y_n \to Y_{n-1} \to \cdots \to Y_1 \to Y_0 := Y
\]
such that $Y_i$ is regular for each $i$ and the morphism $Y_{i+1} \to Y_i$
is the blow-up along a regular subscheme of $Y_i$, for each $i$. Then the
natural maps $\strSh_Y \to \RDer\mu_* \strSh_{Y_n}$ and
$\RDer{\mu}_*{\mu}^!\omega_{Y} \to \omega_{Y}$ are quasi-isomorphisms.
\end{lemma}

\begin{proof}
We first note that the two assertions are equivalent to each
other~\cite[Lemma~(4.2)]{LipCMgraded94}, so we will prove that 
$\strSh_Y \to \RDer\mu_* \strSh_{Y_n}$ is a quasi-isomorphism.
Since $Y$ is normal, the map $\strSh_Y \to \mu_* \strSh_{Y_n}$ is an
isomorphism. Hence we need to show that for every $i > 0$, 
$\RDer^i\mu_* \strSh_{Y_n} = 0$. 
We do this by induction on $n$. Assume that $n=1$.
The question is local on $Y$, so we may
assume that $Y = \Spec S$ for some regular local ring $(S, \frakn)$ and
that $Y_1$ is obtained by the blow-up along an ideal generated by a regular
sequence $x_1, \ldots, x_c \in \frakn \minus \frakn^2$. The Rees algebra
$\Rees_S((x_1, \ldots, x_c))$ is Cohen-Macaulay, so 
$\homology^i(Y_1, \strSh_{Y_1})  = 0$ for every $i > 0$.

Now assume that $n > 1$. Write 
$\nu$ for the composite morphism $Y_n \to Y_1$ and
$\nu'$ for the morphism $Y_1 \to Y$. Then 
$\nu_* \strSh_{Y_n} = \strSh_{Y_1}$ and
$\RDer^i\nu_* \strSh_{Y_n} = 0$ for every $i > 0$. Hence the spectral
sequence (see, e.g.,~\cite[Theorem~5.8.3]{WeibHomAlg94})
\[
E_2^{i,j} = \RDer^j{\nu'}_*\RDer^i{\nu}_* \strSh_{Y_n} \implies
\RDer^{i+j}{\mu}_*\strSh_{Y_n}
\
\]
proves the lemma.
\end{proof}

\subsection*{Pseudo-rational rings}

Lipman~\cite[Section~1a, p.~156]{LipmanDesingTwoDim1978} in dimension
two and Lipman and 
Teissier~\cite[Section~2, p.~102]{LipTeiPseudoRatlSing81} 
more generally defined pseudo-rational rings. Again, as we did with the
definition of $F$-rationality, we present the definition relevant for
excellent rings.

\begin{definition}
A $d$-dimensional local ring $(S, \frakn)$ is said to be
\define{pseudo-rational} if it is Cohen-Macaulay and normal and for every
proper birational morphism $f : Y \to \Spec S$, the map
$\homology^d_{\frakn}(S) \to 
\homology^d_{f^{-1}(\{\frakn\})}(\strSh_Y)$ 
(an edge map in the Leray spectral
sequence $\homology^j_{\frakn}(\RDer^if_*\strSh_Y) \implies 
\homology^{i+j}_{f^{-1}(\{\frakn\})}(\strSh_Y)$) is injective.
\end{definition}

\begin{theorem}[\protect{\cite[Theorem~3.1]{SmithFrational1997}}]
Every excellent $F$-rational local ring is pseudo-rational.
\end{theorem}

\begin{remarkbox}%
\label{remarkbox:twodimpseudorational}
Let $(S, \frakn)$ be two-dimensional pseudo-rational singularity and $Y \to
\Spec S$ a proper birational morphism with $Y$ normal. Write $F$ for the
closed fiber in $Y$, i.e., inverse image of $\{\frakn\}$. Then
$\homology^1_F(Y, \strSh_Y) =
0$~\cite[Theorem~(2.4), p.~177]{LipmanDesingTwoDim1978}. Using duality with
supports~\cite[Theorem, p.~188]{LipmanDesingTwoDim1978}, we see that 
$\homology^1(Y, \omega_Y) = 0$. Additionally, $\homology^1(Y, \strSh_Y) =
0$~\cite[Examples~(a), p.~103]{LipTeiPseudoRatlSing81}. Suppose that 
$Y = \Proj \Rees_S(J)$ for some integrally closed 
$\frakn$-primary ideal $J$. We see from the proof
of~\cite[Corollary~5.4]{LipTeiPseudoRatlSing81} that
$\homology^1(Y, J^n\strSh_Y) = 0$ for every $n \geq 0$.
\end{remarkbox}

\section{Rees algebras}
\label{section:rees}

In this section, we prove some properties of the local cohomology modules
of Rees algebras and extended Rees algebras, which will be used in the
proofs of the theorems.

\begin{discussionbox}
\label{discussionbox:reesExtRees}
Assume that $R$ and $\Rees$ are Cohen-Macaulay. 
We now argue that there exists commutative diagram
\[
\xymatrix{
0 \ar[r] & 
\homology^d_{\Rees_+}(\Rees) \ar[r]^-\phi  \ar@{=}[d] &
\bigoplus\limits_{n<0} \homology^d_{\frakm}(R)t^n \ar[r] \ar@{^(->}[d] &
\homology^{d+1}_{\frakM}(\Rees) \ar[r]  \ar@{^(->}[d]^\gamma &
0 \\
0 \ar[r] & 
\homology^d_{\Rees_+}(\Rees') \ar[r]^-\psi&
\bigoplus\limits_{n \in \ints} \homology^d_{\frakm}(R)t^n \ar[r]&
\homology^{d+1}_{\frakN}(\Rees') \ar[r]  &
0
}
\]
with exact rows. We prove this as follows. The top row is
from~\cite[Lemma~2.7]{HaraWatanabeYoshidaFrationality2002}.
(Note that $\homology^{d}_{\Rees_+}(\Rees)_j=0$ for every $j\geq
0$, by Remark~\ref{remarkbox:CMRees}.)
The lower row is from the exact sequence
\[
\cdots \to
\homology^{i-1}_{\Rees_+}(\Rees'_{t^{-1}}) \to
\homology^{i}_\frakN(\Rees') \to
\homology^{i}_{\Rees_+}(\Rees') \to
\homology^{i}_{\Rees_+}(\Rees'_{t^{-1}}) \to
\homology^{i+1}_\frakN(\Rees') \to
\cdots.
\]
Since $\Rees$ and $R$ are Cohen-Macaulay, $\Rees'$ is Cohen-Macaulay.
We now describe the maps. The left-most vertical map is identity, since
$\Rees'/\Rees$ is a $\Rees_+$-torsion module and $d \geq 2$. The middle
vertical map is the natural inclusion map; the right-most vertical map
$\gamma$ is 
induced by the commutativity of the left square.

To prove the commutativity of the diagram, we need to show that the left
square is commutative. For this, we need to describe the two maps $\phi$
and $\psi$. Write $f = f_1\cdot \cdots \cdot f_d$.
Every homogeneous element of 
$\homology^d_{\Rees_+}(\Rees)$ is the 
class in cohomology of a \tCech cycle 
\[
\frac{at^n}{f^lt^{dl}}
\]
with $a \in \Rees_n$ and some $l$.
Then
\[
\phi\left(\left[\frac{at^n}{f^lt^{dl}}\right]\right) = 
\left[\frac{a}{f^l}\right]\cdot t^{n-dl}.
\]
(See~\cite[Remark, p.~166]{HaraWatanabeYoshidaFrationality2002}). On the
other hand, the map $\psi$ is the map
\[
\homology^{i}_{\Rees_+}(\Rees') \to
\homology^{i}_{\Rees_+}(\Rees'_{t^{-1}}) 
\]
which maps
\[
\left[\frac{at^n}{f^lt^{dl}}\right], \;a \in \Rees'_n
\]
to itself, with $a$ now thought of as an element of 
$(\Rees'_{t^{-1}})_n$.
Hence the left square commutes. Hence the diagram commutes. Also note that
the right-most vertical map is injective. Moreover, for each $j<0$, the
map 
$\homology^{d+1}_{\frakM}(\Rees)_j  \to
\homology^{d+1}_{\frakN}(\Rees')_j$ 
induced by $\gamma$ (i.e., $j$th graded piece of $\gamma$) 
is an isomorphism.
\end{discussionbox}

We now prove some results relating the $F$-rationality of $\Rees$, $\Rees'$
and $X$. To make sense of the following observation, recall that the map
$\gamma$ is bijective in negative degrees.

\begin{lemma}
Adopt the notation of Discussion~\ref{discussionbox:reesExtRees}. Let $0
\neq c \in R$, $e$ a positive integer and $\xi \in 
\homology^{d+1}_{\frakM}(\Rees)$ 
be such that $cF^e(\xi) = 0$. Then 
$cF^e(\gamma(\xi)) = 0$. 
Conversely, let $0
\neq c \in R$, $e$ a positive integer,
$j$ a negative integer, and $\zeta \in 
\homology^{d+1}_{\frakN}(\Rees')_j$ be such that 
$cF^e(\zeta) = 0$. Then 
$cF^e(\gamma^{-1}(\zeta)) = 0$.
\end{lemma}

\begin{proof}
Without loss of generality, $\xi$ is homogeneous. 
Lift it to $\xi'$ in
$\oplus_{n<0}\homology^d_{\frakm}(R)t^n$. Then $cF^e(\xi') \in 
\image(\phi)$. Reading this in the lower row of the diagram of
Discussion~\ref{discussionbox:reesExtRees} we see that 
$cF^e(\gamma(\xi)) = 0$. A similar argument proves the other assertion.
\end{proof}

\begin{remark}
From~\cite[Corollary~1.10]{HaraWatanabeYoshidaFrationality2002}, we know
that there exists $c \in R$ that is a parameter test element for $R$ and
$\Rees$. The same argument shows that 
there exists $c \in R$ that is a parameter test element for $R$, 
$\Rees$ and $\Rees'$.
\end{remark}

As an immediate corollary, we get the following:
\begin{proposition}
\label{proposition:TCReesExtRees}
Adopt the notation of Discussion~\ref{discussionbox:reesExtRees}. Then
\[
\gamma(0^*_{\homology^{d+1}_{\frakM}(\Rees)})
=
\bigoplus_{j<0}\left(0^*_{\homology^{d+1}_{\frakN}(\Rees')}\right)_j.
\]
In particular, 
if $\Rees'_\frakN$ is $F$-rational, then so is $\Rees_\frakM$.
\end{proposition}

\begin{lemma}
\label{lemma:puncRees}
Suppose that $R$ is normal.
Let $S$ be a graded $R$-algebra with
$\Rees\subseteq S \subseteq \overline{\Rees}$. Then
$\sqrt{\frakM S}$ is the homogeneous maximal ideal of $S$. Further,
$\Proj S$ is $F$-rational if and only if 
$\Spec S \minus \{\sqrt{\frakM S}\}$ is $F$-rational. 
\end{lemma}

\begin{proof}
Note that $S_0 = R$. For every $j> 0$, $S_j \subseteq \overline{I^j}t^j$,
so there exists $N$ such that $S_j^N \subseteq I^{jN}t^{jN} \subseteq
\frakM S$. Hence
$\sqrt{\frakM S}$ is the homogeneous maximal ideal of $S$.

Write $X = \Proj S$.
Note that $\sqrt{JtS} = S_+$, so 
$\frakm S + S_+ = \sqrt{\frakM S}$ and
the affine schemes $\Spec S_{f_it}$, $1 \leq i \leq d$, form an open
covering for $\Spec S \minus V(S_+)$ and, by~\cite[Chapitre II,
Corollaire (2.3.14)]{EGA}, $\Spec S_{(f_it)}$, $1 \leq i \leq d$, form
an open covering for $X$. Since $S_{f_it} \simeq
S_{(f_it)}[T, T^{-1}]$~\cite[Chapitre II, (2.2.1)]{EGA}, we see that
$X$ is $F$-rational if and only if 
$\Spec S \minus V(S_+)$ is $F$-rational.
(Use Proposition~\ref{proposition:PolyRationalFrationality}.)

Now assume that 
$\Spec S\minus \{\sqrt{\frakM S}\}$ is $F$-rational. 
Then 
$\Spec S \minus V(S_+)$ is $F$-rational, so $X$ is $F$-rational. 
Conversely assume that 
$X$ is $F$-rational. Let $\frakQ \in
\Spec S \minus \{\sqrt{\frakM S}\}$. If $\frakQ \cap R = \frakm$, then $\frakQ
\in \Spec S \minus V(S_+)$, so $S_\frakQ$ is $F$-rational.
On the other hand, if $\frakQ \cap R \subsetneq \frakm$, then 
there exists $1 \leq i \leq d$ such that $f_i \not \in \frakQ$, so
$S_\frakQ \simeq (S_{f_i})_{\frakQ S_{f_i}}$. Now observe that
$\Spec R\minus \{\frakm\}
\simeq \Proj S \minus V(IS)$
is $F$-rational, so 
$S_{f_i} \simeq R_{f_i}[T]$ is $F$-rational.
\end{proof}

\begin{lemma}
\label{lemma:puncExtRees}
$\Spec \Rees'\minus \{\frakN\}$ is $F$-rational if and only if $R$ 
and $X$ are $F$-rational.
\end{lemma}

\begin{proof}
Since $\frakN = \sqrt{(f_1t, \ldots, f_dt)\Rees' + (t^{-1})\Rees'}$,
we see that 
$\Spec \Rees'\minus \{\frakN\}$ is $F$-rational if and only if
$\Rees'_{t^{-1}}$ and 
$\Rees'_{f_it}$, $1 \leq i \leq d$ are $F$-rational.

Now $\Rees'_{t^{-1}} \simeq R[t, t^{-1}]$, so $\Rees'_{t^{-1}}$ is
$F$-rational if and only if $R$ is $F$-rational. Moreover, every element of
the $\Rees$-module $\Rees'/\Rees$ is annihilated by a power of $\Rees_+$,
so the inclusion $\Rees_{f_it} \to \Rees'_{f_it}$ is an isomorphism for
every $1 \leq i \leq d$. Hence, arguing as in the proof of
Lemma~\ref{lemma:puncRees}, we conclude that $X$ if $F$-rational if and
only if $\Rees'_{f_it}$ is $F$-rational for every $1 \leq i \leq d$.
\end{proof}

\section{Theorem~\protect{\ref{theorem:extReesAndRees}}}

In this section we prove 
Theorem~\ref{theorem:extReesAndRees} and some corollaries.

\begin{lemma}
\label{lemma:ReesCM}
If $R$ is $F$-rational and $\Rees'$ is Cohen-Macaulay, 
$\Rees$ is Cohen-Macaulay. 
\end{lemma}

\begin{proof}
Since $\Rees'$ is Cohen-Macaulay, $G \simeq \Rees'/(t^{-1})$ 
is Cohen-Macaulay. Note that $R$ is pseudo-rational. Hence the lemma
follows from~\cite[Theorem~(5)]{LipCMgraded94}.
\end{proof}

\begin{lemma}
\label{lemma:RReesFratlTCExtRees}
Suppose that $R$ is Cohen-Macaulay and that $\Rees$ is $F$-rational. 
Then $R$ is $F$-rational if and only if 
$0^*_{\homology^{d+1}_\frakN(\Rees')} = 0$.
\end{lemma}

\begin{proof}
We see
from the lower row of the
diagram in Discussion~\ref{discussionbox:reesExtRees} that 
${\homology^{d}_\frakm(R)}  \subseteq 
{\homology^{d+1}_\frakN(\Rees')}_0$. Hence, 
if $0^*_{\homology^{d+1}_\frakN(\Rees')} = 0$, then
$0^*_{\homology^{d}_\frakm(R)} \subseteq
\left(0^*_{\homology^{d+1}_\frakN(\Rees')}\right)_0 = 0$,
and, therefore, $R$ is $F$-rational.

Conversely, assume that $R$ is $F$-rational. Every homogeneous element of
${\homology^{d+1}_\frakN(\Rees')}$ is the class of a \tCech cycle of the
form
\[
\frac{at^k}{f^lt^{(d-1)l}}
\]
with $at^k \in \Rees'_k$ and $f = f_1\cdots f_d$. (We use the
$\frakN$-primary ideal $(f_1t, \ldots, f_dt, t^{-1})$ for describing
${\homology^{d+1}_\frakN(\Rees')}$.) Suppose that
\[
\left[\frac{at^k}{f^lt^{(d-1)l}}\right] \in 
0^*_{\homology^{d+1}_\frakN(\Rees')}.
\]
Then from the lower row of the commutative diagram in
Discussion~\ref{discussionbox:reesExtRees}, we see that $k \geq (d-1)l$, so
this element of ${\homology^{d+1}_\frakN(\Rees')}$ is the image of 
\[
\left[\frac{a}{f^l}\right]\cdot t^{k-(d-1)l} \in 
\oplus_{n \geq 0} \homology^d_\frakm(R) t^n.
\]
Since there exists $c \in R$ such that 
\[
\left[\frac{ca^qt^{kq}}{f^{lq}t^{(d-1)lq}}\right] = 0
\]
we conclude that 
\[
\left[\frac{ca^q}{f^{lq}}\right]\cdot t^{(k-(d-1)l)q} = 0 \in 
\oplus_{n \geq 0} \homology^d_\frakm(R) t^n.
\]
so 
\[
\left[\frac{ca^q}{f^{lq}}\right] = 0 \in \homology^d_\frakm(R).
\]
Since $R$ is $F$-rational, 
$0^*_{\homology^{d+1}_\frakN(\Rees')} = 0$.
\end{proof}

\begin{proof}%
[{Proof of Theorem~\protect{\ref{theorem:extReesAndRees}}}]
\underline{\eqref{theorem:extReesAndRees:ExtReesFratl} $\implies$
\eqref{theorem:extReesAndRees:ReesFratl}}:
$\Spec \Rees \minus \{\frakM\}$ is $F$-rational by 
Lemmas~\ref{lemma:puncExtRees} and~\ref{lemma:puncRees}.
On the other hand, $\Rees_\frakM$ is $F$-rational by 
Lemma~\ref{lemma:ReesCM} and
Proposition~\ref{proposition:TCReesExtRees}.

\underline{ 
\eqref{theorem:extReesAndRees:ReesFratl}
$\implies$
\eqref{theorem:extReesAndRees:ExtReesFratl}
}:
$\Spec \Rees' \minus \{\frakN\}$ is $F$-rational by 
Lemmas~\ref{lemma:puncExtRees} and~\ref{lemma:puncRees}.
Lemma~\ref{lemma:RReesFratlTCExtRees} implies that $\Rees'_\frakN$ is
$F$-rational.
\end{proof}

In general, even if $\Rees$ is $F$-rational, $R$ need not be
$F$-rational.
Hara, Watanabe and Yoshida show that if 
$\Rees$ is $F$-rational and 
the $a$-invariant of the associated
graded ring $G$ is at most $-2$
then $R$ is
$F$-rational~\cite[Corollary~2.13]{HaraWatanabeYoshidaFrationality2002}. 
In~\cite[Example~3.9]{HaraWatanabeYoshidaFrationality2002}, they exhibit an
example where $R$ is not $F$-rational but 
$\Rees$ is. 
As an application of our earlier results,
we get a sufficient condition for the $F$-rationality
of $R$, given the $F$-rationality of $\Rees'$,
which
recovers~\cite[Corollary~2.13]{HaraWatanabeYoshidaFrationality2002}.

\begin{proposition}
\label{theorem:assGrFInj}
Let $(R, \frakm)$ be an excellent Cohen-Macaulay local 
domain of prime characteristic and $I$ an $\frakm$-primary
$R$-ideal. 
Suppose that 
$\Rees_R(I)$ is $F$-rational.
If the Frobenius map $\homology^d_{G_+}(G)_{-1} \to
\homology^d_{G_+}(G)_{-p}$ is
injective, then $R$ and $\Rees'_R(I)$ are $F$-rational.
\end{proposition}

\begin{proof}
We first show that $0^*_{\homology^{d+1}_{\frakN}(\Rees')} = 0$. 
By way of contradiction, assume that 
$0^*_{\homology^{d+1}_{\frakN}(\Rees')} \neq 0$. 
By Proposition~\ref{proposition:TCReesExtRees}
we may pick a non-zero $\xi \in 
0^*_{\homology^{d+1}_{\frakN}(\Rees')}$ of smallest degree. 
Write $j+1$ for this degree.
Then $j \geq -1$. Note that $t^{-1}\xi = 0$, so
$\xi$ is the image of a non-zero element $\zeta \in
\homology^{d}_{G_+}(G)_{j}$. Since $\Rees$ is Cohen-Macaulay, $j <
0$.  Hence $j =-1$.

Now consider the diagram~\eqref{equation:FrobModNZD} from
Discussion~\ref{discussionbox:FrobOnIdeals}, with $S = \Rees'$, $x =
t^{-1}$ and $m = -1$. Rewriting it with only the relevant cohomology
groups we get the following:
\[
\xymatrix{
0 \ar[r] &
\homology^{d}_{G_+}(G)_{-1} \ar[r] \ar[d]^F&
\homology^{d+1}_{\frakN}(\Rees')_{0} \ar[r]^{t^{-1}} \ar[d]^{t^{1-p}F}&
\homology^{d+1}_{\frakN}(\Rees')_{{-1}} \ar[r] \ar[d]^F&
0 \\
0 \ar[r] &
\homology^{d}_{G_+}(G)_{-p} \ar[r] &
\homology^{d+1}_{\frakN}(\Rees')_{1-p} \ar[r]^{t^{-1}}&
\homology^{d+1}_{\frakN}(\Rees')_{-p} \ar[r]&
0.
}
\]
Note that $F(\xi) \in 
\left(0^*_{\homology^{d+1}_{\frakN}(\Rees')}\right)_0$, so 
$t^{1-p}F(\xi) = 0$. On the other hand, by hypothesis, $F(\zeta) \neq 0$, a
contradiction. Hence
$0^*_{\homology^{d+1}_{\frakN}(\Rees')} = 0$.

Hence $\Rees'_\frakN$ and $R$ are $F$-rational 
(Lemma~\ref{lemma:RReesFratlTCExtRees}). 
Since $\Proj \Rees$ is $F$-rational, $\Spec \Rees' \minus \{\frakN\}$ is
$F$-rational (Lemma~\ref{lemma:puncExtRees}). Therefore $\Rees'$ is
$F$-rational.
\end{proof}

The next example shows that the hypothesis on 
$\homology^d_{G_+}(G)$ 
in Proposition~\ref{theorem:assGrFInj} is 
only sufficient, but not necessary.

\begin{examplebox}
Let $R=\Bbbk[[X,Y,Z]]/(X^2+Y^3+Z^5)$, with $\Bbbk$ a field of
characteristic $p \geq 7$.
It is a two-dimensional $F$-rational ring. Since $\frakm = (X,Y,Z)$ is
integrally closed, the Rees algebra $\Rees_R(\frakm)$ is $F$-rational 
by~\cite[Theorem~3.1]{HaraWatanabeYoshidaFrationality2002}
(Theorem~\ref{theorem:dimtwo} below).
Note that $G=\Bbbk[x,y,z]/(x^2)$, where $x, y, z$ correspond to
$X,Y,Z$. Compute 
$H^2_{G_+}(G)$ from the \tCech complex
$\check{C}^{\bullet}(y,z;G)$. 
Use Remark~\ref{remarkbox:cechcyclezero} to conclude that
the class 
\[
\left[\frac{x}{yz}\right] \in H^2_{G_+}(G)_{-1}
\]
of the \tCech cycle $\frac{x}{yz}$ is non-zero, but 
\[
F\left(\left[\frac{x}{yz}\right]\right) =
\left[\frac{x^p}{y^pz^p}\right] = 
0 \in H^2_{G_+}(G)_{-p}.
\]
Hence the map $H^2_{G_+}(G)_{-1}\to H^2_{G_+}(G)_{-p}$ is not injective.
\end{examplebox}

As another application of our results
we recover the following result of Hara, Watanabe and
Yoshida~\cite[Theorem~3.1]{HaraWatanabeYoshidaFrationality2002}.
(They mention that Huneke and Smith also had obtained this result.)

\begin{theorem}
\label{theorem:dimtwo}
Let $(R,\frakm)$ be a two-dimensional excellent 
$F$-rational domain of prime characteristic. 
Let $I$ be an $\frakm$-primary
integrally closed ideal. Then $\Rees_R(I)$ is $F$-rational.
\end{theorem}

\begin{remark}
\label{remark:ReesCMinDimTwo}
By~\cite[Example~(3)]{LipCMgraded94}, we know that $\Rees$ is
Cohen-Macaulay. Hence $G$ and $\Rees'$ are Cohen-Macaulay.
\end{remark}

\begin{proof}[Proof of Theorem~\protect{\ref{theorem:dimtwo}}]
We claim that the ideal 
$(f_1t, f_2t, t^{-1})$ of $\Rees'$ is tightly closed. Assume the claim.
Then, by~\cite[Proposition~4.14]{HochsterHunekeTCInvThyBSThm90}, 
$(f_1t, f_2t, t^{-1})\Rees'_{\frakN}$ is tightly closed. Since 
$\Rees'_{\frakN}$ is Cohen-Macaulay, it is
$F$-rational~\cite[Theorem~4.2]{HochsterHunekeFregSmBsch94}. 
Discussion~\ref{discussionbox:reesExtRees} and 
Proposition~\ref{proposition:TCReesExtRees} apply to $\Rees$, so 
$\Rees_\frakM$ is $F$-rational.
By~\cite[Lemma~1.12]{HaraWatanabeYoshidaFrationality2002}, $\Rees$ is
$F$-rational.

Now to prove the claim, consider an element 
$at^k \in (f_1t, f_2t, t^{-1})^*$, with $a \in I^k$. (It suffices to
consider homogeneous elements.)
If $k < 0$, then $at^k \in (t^{-1})$, so we may
assume that $k \geq 0$. 
If $k \geq 2$, then 
$I^k = JI^{k-1}$~\cite[Corollary~5.4]{LipTeiPseudoRatlSing81}, 
so $at^k \in (f_1t, f_2t)$. Therefore
we need to consider $k=0$ and $k=1$.
For every $0 \neq c \in I$, $\Rees'_c \simeq R_c[t,
t^{-1}]$ is $F$-rational, so we may pick a parameter test element $c$
from $I$.  If $k=0$, then for all sufficiently large $q = p^e$, 
there exist $\alpha_q, \beta_q \in R$ and $\gamma_q \in I^q$ such that
$ca^q  = \alpha_q t^{-q} f_1^q t^q + 
\beta_q t^{-q} f_2^q t^q + 
\gamma_q t^{q} t^{-q} \in (f_1^q, f_2^q) + I^q \subseteq I^q$. Since $I$
is integrally closed, $a \in I$. Since $I = (t^{-1})_0$, $a \in (t^{-1})$.
If $k = 1$, then for all sufficiently large $q = p^e$, 
there exist $\alpha_q, \beta_q \in R$ and $\gamma_q \in I^{2q}$ such
that $ca^q t^q = \alpha_q f_1^q t^q + 
\beta_q f_2^q t^q + 
\gamma_q t^{2q} t^{-q}$, so $ca^q 
\in (f_1^q, f_2^q) + I^{2q}
\subseteq (f_1^q, f_2^q)$, where the last inclusion follows from the
pseudo-rationality of $R$~\cite[Corollary~5.4]{LipTeiPseudoRatlSing81}. 
Since $R$ is $F$-rational, $a \in (f_1, f_2)$, so 
$at \in (f_1t, f_2t)$.
\end{proof}

\section{Theorem~\protect{\ref{theorem:rationalblowup}}}

\begin{notation}
Let $S = \bigoplus_{n \in \naturals}S_n$ be a graded ring with $S_0$ 
a local ring with maximal ideal $\frakn_0$. Write $\frakn$ for the
homogeneous maximal ideal $\frakn_0S + S_+$.
For a positive integer $N$, we write $S^{(N)} := \oplus_{n
\geq 0} S_{Nn}$ and $\frakn^{(N)}$ for the homogeneous maximal ideal of
$S^{(N)}$. 
\end{notation}

\begin{observationbox}
\label{observationbox:tcOfZeroVeronese}
Note that for integers $N>0$ and $i \geq 0$, 
\[
\homology^{i}_{\frakn^{(N)}}(S^{(N)}) = 
\bigoplus_{n \in \ints} 
\left(\homology^{i}_{\frakn}(S)\right)_{Nn}.
\]
Write $W  = 0^*_{\homology^{i}_{\frakn}(S)}$. It is a graded $S$-module.
Let $m \in \ints$ and $\xi \in W_m$. Then we may assume that there exists
$c \in S^{(m)}$ such that $c\xi^q = 0$ for every sufficiently large power
$q$ of $p$. Hence 
\begin{align*}
\bigoplus_{n \in \ints} W_{Nn} & \subseteq
0^*_{\homology^{i}_{\frakn^{(N)}}(S^{(N)})}.
\intertext{On the other hand, we can check directly that}
0^*_{\homology^{i}_{\frakn^{(N)}}(S^{(N)})} & \subseteq 
\bigoplus_{n \in \ints} W_{Nn},  
\intertext{so}
0^*_{\homology^{i}_{\frakn^{(N)}}(S^{(N)})}& =
\bigoplus_{n \in \ints} W_{Nn}. 
\qedhere
\end{align*}
\end{observationbox}

One of the ingredients of the proof of 
Theorem~\ref{theorem:rationalblowup}
is the following
proposition, which is analogous to the results of
E.~Hyry~\cite{HyryBlowupRingsRationalSings1999} relating the rationality
$\Proj \Rees_R(I)$ to that of $\Rees_R(I)$ in characteristic zero.
\begin{proposition}
\label{proposition:veroneseOfRees}
Let $(R, \frakm)$ be an excellent noetherian local normal
Cohen-Macaulay domain of prime characteristic and $I$ an $\frakm$-primary
$R$-ideal. Let $S$ be a graded $R$-algebra with 
${\Rees_{R}(I)} \subseteq S \subseteq \overline{\Rees_{R}(I)}$. 
Write $X = \Proj S$.  Suppose that $X$ is $F$-rational and that 
$\homology^i(X, \strSh_X) = 0$ for every $i \geq 1$.
Then for every integer $N \gg 0$, the subring $\oplus_{n \geq 0} S_{Nn}$ is
$F$-rational.
\end{proposition}

\begin{lemma}
\label{lemma:veroneseReesFiniteLengthTightClosureZero}
With the hypothesis in Proposition~\ref{proposition:veroneseOfRees}, 
$0^*_{\homology^{d+1}_{\frakM}(S)}$ is a finite-length $S$-module.
In particular, for every integer $N \gg 0$, 
$0^*_{\homology^{d+1}_{\frakM^{(N)}}(S^{(N)})} = 0$.
\end{lemma}

\begin{proof}
From Lemma~\ref{lemma:puncRees}
we see that 
$\Spec S \minus \{\sqrt{\frakM S}\}$ is $F$-rational. Now apply
Observation~\ref{observationbox:finiteLengthTCOfZero} to prove 
the first assertion. The second assertion follows from this
and Observation~\ref{observationbox:tcOfZeroVeronese}, after noting that
$\homology^{d+1}_\frakM(S)_n = 0$ for every $n \geq 0$.
\end{proof}

\begin{proof}%
[Proof of Proposition~\protect{\ref{proposition:veroneseOfRees}}]
Write $\frakn = \sqrt{\frakM S}$.
For every integer $N \gg 0$, $S^{(N)}$ is Cohen-Macaulay
by~\cite[Theorem~(4.1)]{LipCMgraded94} and 
$0^*_{\homology^{d+1}_{\frakn^{(N)}}(S^{(N)})} = 0$
by Lemma~\ref{lemma:veroneseReesFiniteLengthTightClosureZero}.
Hence $\left(S^{(N)}\right)_{\frakn^{(N)}}$ is $F$-rational
by~\cite[Proposition 4.4(ii)]{SmithTestIdealsInLocalRings1995}.
On the other hand, 
by Lemma~\ref{lemma:puncRees},
$\Spec S^{(N)} \minus {\frakn^{(N)}}$ is $F$-rational
since $\Proj S^{(N)} = \Proj S$ is $F$-rational.
\end{proof}

We now prove some lemmas needed in the proof of 
Theorem~\ref{theorem:rationalblowup}. 
For the remainder of this section we assume the hypothesis of the theorem
and 
adopt the following notation: Write $X
= \Proj S$ and $f$ 
for the structure morphism $X \to \Spec R$.
Since $R$ is a rational singularity, it is
Cohen-Macaulay. Write $\omega_R$ for a dualizing module for $R$, which
exists since $R$ is of finite type over a field and a Cohen-Macaulay
ring. Write $\omega_X$ for the
left-most non-zero cohomology sheaf of the complex $f^!\omega_R$; since $X$
if $F$-rational (and, \textit{a fortiori}, Cohen-Macaulay), it is a
dualizing sheaf for $X$.

\begin{lemma}
\label{lemma:desingDominatesX}
There exists a desingularization $g : Y \to X$ such that 
$\homology^0(Y, \omega_Y) = \omega_R$ and
$\homology^i(Y, \omega_Y) = 
\homology^i(Y, \strSh_Y) = 0$ for every $i > 0$.
\end{lemma}

\begin{proof}
Since $R$ is normal, $\homology^0(Y, \strSh_Y) = R$, so 
by ~\cite[Lemma~(4.2)]{LipCMgraded94}
the assertions
about $\omega_Y$ would follow from showing that 
$\homology^i(Y, \strSh_Y) = 0$ for every $i > 0$.
Since $R$ is a rational singularity, pick a desingularization $\mu : Y \to \Spec R$ such that the natural maps $R \to \RDer\mu_*\strSh_Y$ and
$\RDer\mu_*\mu^!\omega_R \to \omega_R$ are quasi-isomorphisms. By
Proposition~\ref{proposition:principalization}, 
we can find a proper birational morphism $\mu' : Y'
\to Y$ such that the ideal sheaf 
$I\strSh_{Y'}$ is invertible.
Now use Lemma~\ref{lemma:finiteSeqBlowupsRDer} to conclude 
that the map $\strSh_Y \to \RDer\mu'_* \strSh_{Y'}$ is a
quasi-isomorphism. Therefore 
the spectral sequence
\[
E_2^{i,j} = \RDer^j{\mu}_*\RDer^i{\mu'}_* \strSh_{Y'} \implies
\RDer^{i+j}({\mu}{\mu'})_*\strSh_{Y'},
\]
helps us conclude that 
$\homology^i(Y', \strSh_{Y'}) = 0$ for every $i > 0$.

Since $Y'$ is normal and 
$I\strSh_{Y'}$ is invertible the map 
$(\mu\mu') : Y' \to \Spec R$ factors as $fg$ for some $g : Y' \to X$.
Relabel $Y'$ as $Y$.
\end{proof}

\begin{lemma}
\label{lemma:higherDirImageZeroDiml}
Let $g : Y \to X$ be as in Lemma~\ref {lemma:desingDominatesX}. 
Then
\begin{asparaenum}
\item 
\label{lemma:higherDirImageZeroDiml:directImage}
The maps $\strSh_X \to g_* \strSh_Y$ and $g_* \omega_{Y} \to \omega_X$ are
isomorphisms. In particular, $\homology^0(X, \strSh_X) = R$ and 
$\homology^0(X, \omega_X) = \omega_R$.

\item 
\label{lemma:higherDirImageZeroDiml:higherdirectImage}
For every $i \geq 1$, 
$\RDer^ig_* \strSh_{Y}$ and $\RDer^ig_* \omega_{Y}$ have
zero-dimensional support, if they are non-zero.
\end{asparaenum}
\end{lemma}

\begin{proof}

\underline{\eqref{lemma:higherDirImageZeroDiml:directImage}}:
The map $\strSh_X \to g_* \strSh_Y$ is an isomorphism since $X$ is normal.
That the map $g_* \omega_{Y} \to \omega_X$ is an isomorphism 
follows from the characterization of pseudo-rational
rings in~\cite[Corollary (of (iii)), p.~107]{LipTeiPseudoRatlSing81}, after
noting that $X$ is  
pseudo-rational~\cite[Theorem~3.1]{SmithFrational1997}. The remaining
assertions follow from the above isomorphisms and
Lemma~\ref{lemma:desingDominatesX}.

\underline{\eqref{lemma:higherDirImageZeroDiml:higherdirectImage}}:
Let $x \in X$ be a point with $\dim \strSh_{X,x} = 2$. 
Write $Y_1 = Y \times_X \Spec \strSh_{X,x}$ and consider the
cartesian square
\[
\xymatrix{
Y_1  \ar[r] \ar[d]_{g_1} & Y \ar[d]_{g} \\
\Spec \strSh_{X,x} \ar[r] & X
}
\]
Then $g_1 : Y_1 \to \Spec \strSh_{X,x}$ is a desingularization of the
two-dimensional pseudo-rational ring $\strSh_{X,x}$; by
Remark~\ref{remarkbox:twodimpseudorational},
$\RDer^1{g_1}_* \strSh_{Y_1} = \RDer^1{g_1}_* \omega_{Y_1} = 0$.
Since $\dim \strSh_{X,x} = 2$, 
$\RDer^i{g_1}_* \strSh_{Y_1} = \RDer^i{g_1}_* \omega_{Y_1} = 0$
for every $i \geq 2$. By flat base-change
$(\RDer^ig_* \strSh_{Y})_x = 0= (\RDer^ig_* \omega_{Y})_x$, for every $i
\geq 1$. 
(Since $\omega_Y$ is defined using $g^!$, we need to 
use~\cite[Theorem~2, p.~394]{VerdierBaseChangeTwistedInverseImage69}
for the second equality.) Hence these sheaves have zero-dimensional
support, if they are non-zero.
\end{proof}

\begin{proof}[Proof of Theorem~\ref{theorem:rationalblowup}]
In view of Proposition~\ref{proposition:veroneseOfRees}, it suffices to
show that $\homology^i(X, \strSh_X) = 0$ for every $i > 0$.
Let $g : Y \to X$ be as in Lemma~\ref{lemma:desingDominatesX}.
Write $h = fg$.
Let $\calF$ be $\strSh_Y$ or $\omega_Y$. 
Then in the spectral sequence
\[
E_2^{i,j} = \homology^j(X, \RDer^ig_* \calF)
=
\RDer^jf_*\RDer^ig_* \calF \implies
\RDer^{i+j}h_*\calF = \homology^{i+j}(Y, \calF),
\]
$E_2^{i,j} = 0$ if $i>0$ and $j>0$
by
Lemma~\ref{lemma:higherDirImageZeroDiml}\eqref{lemma:higherDirImageZeroDiml:higherdirectImage}.
This gives an exact sequence
{\small
\[
0 \rightarrow 
\homology^1(X, g_*\calF) \rightarrow
\homology^1(Y, \calF) \rightarrow
\homology^0(X, \RDer^1g_* \calF) \rightarrow 
\homology^2(X, g_*\calF) \rightarrow
\homology^2(Y, \calF)\rightarrow
\homology^0(X, \RDer^2g_*\calF) \rightarrow 0.
\]}
Use Lemma~\ref{lemma:desingDominatesX} to conclude that 
$\homology^1(X, g_*\calF) = 0$, 
that the map 
$\homology^0(X, \RDer^1g_* \calF) \to \homology^2(X, g_*\calF)$ 
is an isomorphism, and (since $\RDer^2g_*\calF$ has zero-dimensional
support) that $\RDer^2g_*\calF = 0$.

We now have the following situation: 
$\homology^0(X, \strSh_X) = R$, 
$\homology^1(X, \strSh_X) = 0$, 
$\homology^0(X, \omega_X) = \omega_R$  and
$\homology^1(X, \omega_X) = 0$. Note, also, that 
$\homology^2(X, \strSh_X)$ is a finite-length $R$-module.
We need to show that $\homology^2(X, \strSh_X)=0$.

We apply duality for proper 
morphisms~\cite[Chapter VII,~Theorem~3.3, p.~379]{HartRD66} to $f$ to see
that
\[
\RDer f_* \omega_X \simeq \RDer \Hom_R(\RDer f_*\strSh_X, \omega_R).
\]
The right-side can be computed using a second-quadrant 
spectral sequence with 
\[
E_2^{-i,j} = \Ext_R^j(\homology^i(X, \strSh_X), \omega_R) \implies
\homology^{i+j}(X,\omega_X). 
\]
Since the only possibly non-zero terms are $E_2^{0,0} =\omega_R$ 
and $E_2^{-2,3} = \Ext_R^3(\homology^2(X, \strSh_X), \omega_R)$, we
get an exact sequence (with $\alpha$ coming from the $E_3$-page)
\[
0 \to \homology^0(X, \omega_X) \to
\omega_R \stackrel{\alpha}{\to} 
\Ext_R^3(\homology^2(X, \strSh_X), \omega_R)
\to \homology^1(X, \omega_X) \to 0
\]
Since $\homology^0(X, \omega_X) \simeq \omega_R$, we conclude that $\alpha
= 0$, so
$\Ext_R^3(\homology^2(X, \strSh_X), \omega_R) \simeq 
\homology^1(X, \omega_X)$.
Hence $\homology^2(X, \strSh_X) = 0$.
\end{proof}

The above proof also gives the following proposition:

\begin{proposition}
$X$ has rational singularities.
\end{proposition}

\begin{proof}
For $\calF = \strSh_Y$ or 
$\calF = \omega_Y$, we established the following:
\begin{inparaenum}
\item $\RDer^2g_*\calF = 0$;
\item $\RDer^1g_*\calF$ has zero-dimensional support, if it is non-zero;
\item $\homology^0(X, \RDer^1g_*\calF) = \homology^2(X, \strSh_X) = 0$.
\end{inparaenum}
Hence $\RDer^1g_*\calF=0$.
\end{proof}

\def\cfudot#1{\ifmmode\setbox7\hbox{$\accent"5E#1$}\else
  \setbox7\hbox{\accent"5E#1}\penalty 10000\relax\fi\raise 1\ht7
  \hbox{\raise.1ex\hbox to 1\wd7{\hss.\hss}}\penalty 10000 \hskip-1\wd7\penalty
  10000\box7}

\end{document}